\documentclass[12pt,a4paper,reqno]{amsart}
\usepackage{amssymb,latexsym,amsfonts,amsthm,upref,amsmath,capt-of}
\usepackage[english]{babel}
\usepackage[margin=1in]{geometry} 
\usepackage[foot]{amsaddr}
\usepackage[dvipsnames,x11names]{xcolor}
\definecolor{myblue}{HTML}{003388}
\definecolor{mygreen}{HTML}{338800}
\definecolor{myred}{HTML}{880033}
\usepackage{hyperref}   
\hypersetup{colorlinks,citecolor=mygreen,filecolor=black,linkcolor=myred,urlcolor=myblue}
\usepackage{comment} 
\usepackage{xcolor}
\usepackage{enumerate} 
\usepackage{tikz}
\usetikzlibrary{chains}
\usepackage{float}
\usepackage{cleveref}
\usepackage{mathtools}
\usepackage{cite}
\usepackage{filecontents}

\newtheorem{thm}{Theorem}[section]
\newtheorem{lem}{Lemma}[subsection]

\newtheorem*{thm*}{Theorem}
\newtheorem{example}{Example}[subsection]

\newcommand{\beq}{\begin{equation}}
\newcommand{\eeq}{\end{equation}}
\newcommand{\ch}{\mathop{\mathrm{ch}}}
\newcommand{\Cc}{\mathcal{C}}
\newcommand{\Ec}{\mathcal{E}}
\newcommand{\Dc}{\mathcal{D}}

\numberwithin{equation}{section}

\makeatletter
\def\author@andify{%
  \nxandlist {\unskip ,\penalty-1 \space\ignorespaces}%
    {\unskip {} \@@and~}%
    {\unskip \penalty-2 \space \@@and~}%
}
\makeatother

\tikzset{node distance=2em, ch/.style={circle,draw,on chain,inner sep=2pt},chj/.style={ch,join},every path/.style={shorten >=4pt,shorten <=4pt},line width=1pt,baseline=-1ex}

\let\dlabel=\alabel

\newcommand{\dnode}[2][chj]{%
\node[#1,label={below:\dlabel{#2}}] {};
}

\newcommand{\dnodenj}[1]{%
\dnode[ch]{#1}
}

\newcommand{\dydots}{%
\node[chj,draw=none,inner sep=1pt] {\dots};
}

\begin{document}

\title{A note on principal subspaces of the affine Lie algebras in types $B_l^{(1)}$, $C_l^{(1)}$, $F_4^{(1)}$ and $G_2^{(1)}$}

\author{Marijana Butorac}

\address{Department of Mathematics, University of Rijeka, Radmile Matej\v{c}i\'{c} 2, 51\,000 Rijeka, Croatia}
\email{mbutorac@math.uniri.hr}



\begin{abstract} 
We  construct  quasi-particle bases of principal subspaces of standard modules $L(\Lambda)$, where $\Lambda=k_0\Lambda_0+k_j\Lambda_j$, and $\Lambda_j$ denotes the fundamental weight of affine Lie algebras of type  $B_l^{(1)}$, $C_l^{(1)}$, $F_4^{(1)}$ or $G_2^{(1)}$ of level one.
From the given bases we find characters of principal subspaces.
\end{abstract}

\maketitle

\section*{Introduction}
This paper is a continuation of our study  \cite{Bu1,Bu2,Bu3,BK} of the principal subspaces associated to the standard module $L(k\Lambda_0)$, for $k \geq 1$, of non-simply laced affine Lie algebras of type  $B_l^{(1)}$, $C_l^{(1)}$, $F_4^{(1)}$ and $G_2^{(1)}$. In \cite{G1}, G. Georgiev constructed quasi-particle bases of  principal subspaces of standard modules $L(\Lambda)$ with the rectangular highest weight, that is, $\Lambda$ is of the form $k_0\Lambda_0+k_j\Lambda_j$, and $\Lambda_j$ denotes the fundamental weight of level one, in the case of affine Lie algebra of type $A_l^{(1)}$. In \cite{BK}, we extended Georgiev's approach to the principal subspaces which correspond to rectangular weights of affine Lie algebras of type $D_l^{(1)}$, $E_6^{(1)}$, $E_7^{(1)}$ and $E_8^{(1)}$.

The main result of this work is the construction of quasi-particle bases of principal subspaces of standard modules $L(\Lambda)$ with the rectangular highest weight for the remaining cases of untwisted affine Lie algebras (Theorem \ref{t1}). From the constructed bases we find characters of principal subspaces  (Theorem \ref{thm_karakter}). Obtained characters of principal subspaces are connected with the characters of parafermio\-nic field theories (see \cite{AKS, FS,GG,G2,Gep,KNS}). This connection is further studied in the paper \cite{BKP}.

Our construction follows closely the construction of quasi-particle bases of principal subspaces of vacuum standard modules $L(k\Lambda_0)$  in \cite{Bu1,Bu2,Bu3,BK,G1}. The starting point in this construction is to find all relations among quasi-particles, which are then used to find the spanning sets. In this paper we use these results to construct the spanning sets of principal subspaces of $L(\Lambda)$. The main difference with the case of the principal subspace of $L(k\Lambda_0)$ is in the formulation of initial conditions (Lemma \ref{pocuv1} and Lemma \ref{pocuv2}) in terms of quasi-particles. 

The main idea of the proof of linear independence  is, as in\cite{Bu1,Bu2,Bu3,BK,G1}, from a finite linear combination $\sum_{a \in A} c_ab_av=0$  of quasi-particle monomial vectors $b_av$ from the spanning set, obtain the following linear combination $\sum_{a \in A} c_ab'_av=0$, where $b'_av$ are still from the spanning set, such that $b_a < b'_a$, with the respect to the linear order on quasi-particles. To do this we use coefficients of intertwining operators for vertex operator algebra $L(\Lambda_0)$ associated with the affine Lie algebra from \cite{Li1,Li2,Li3}, together with simple current maps in the case of affine Lie algebras of type $B_l^{(1)}$ and $C_l^{(1)}$, and Weyl group translation operators among standard modules of level $1$.

\section{Preliminaries}\label{prelim}
Let $\mathfrak{g}$ be a complex simple Lie algebra of type $B_l$, $C_l$, $F_4$ or $G_2$ with the triangular decomposition $\mathfrak{g} =\mathfrak{n}_{-}\oplus \mathfrak{h}\oplus \mathfrak{n}_{+}$, where $\mathfrak{h}$ denotes the Cartan subalgebra of $\mathfrak{g}$. Let $\left< \cdot, \cdot \right>$ be the invariant symmetric nondegenerate bilinear form on $\mathfrak{g}$ normalized so that long roots have length $\sqrt{2}$. Denote by $R_{+} \subset \mathfrak{h}^{\ast}$ the set of positive roots of $\mathfrak{g}$, by $R$ the set of roots, by $Q$ the root lattice and by $\{\alpha_1, \ldots , \alpha_l\}$ the subset of simple roots. We fix the standard choice of simple roots, which we now recall. Denote by $\{\epsilon_1, \ldots , \epsilon_l \}$ the usual orthonormal basis of the $\mathbb{R}^l$. Then in the case of $B_l$, we have the following set of simple roots
$$\Big\{\alpha_1= \epsilon_1-\epsilon_2  , \ldots, \alpha_{l-1}=\epsilon_{l-1}-\epsilon_l ,\alpha_l= \epsilon_l\Big\},$$  
which correspond to the following labeling of Dynkin diagram
\begin{center}\begin{tikzpicture}[start chain]
\dnode{1}
\dnode{2}
\dydots
\dnode{l-1}
\dnodenj{l}
\path (chain-4) -- node[anchor=mid] {\(\Rightarrow\)} (chain-5);
\end{tikzpicture} \ .\end{center}
In the case of $C_l$, we will use the following notation for the basis of the root system
$$\Big\{\alpha_1=\sqrt{2}\epsilon_l , \alpha_2=\frac{1}{\sqrt{2}}\left(\epsilon_{l-1}-\epsilon_l\right),\ldots,\alpha_{l-1}=\frac{1}{\sqrt{2}}\left(\epsilon_{2}-\epsilon_3\right),\alpha_l=\frac{1}{\sqrt{2}}\left(\epsilon_1-\epsilon_2\right)\Big\}, $$ 
so that we have the following labeling of the Dynkin diagram 
\begin{center}\begin{tikzpicture}[start chain]
\dnode{l}
\dnode{l-1}
\dydots
\dnode{2}
\dnodenj{1}
\path (chain-4) -- node[anchor=mid] {\(\Leftarrow\)} (chain-5);
\end{tikzpicture} \ .\end{center}
In the case of $F_4$ we have 
 \begin{center}\begin{tikzpicture}[start chain]
\dnode{1}
\dnode{2}
\dnodenj{3}
\dnode{4}
\path (chain-2) -- node[anchor=mid] {\(\Rightarrow\)} (chain-3);
\end{tikzpicture} \ ,
\end{center}
where $$\Big\{ \alpha_1=\epsilon_2 -\epsilon_3, \alpha_2=\epsilon_3-\epsilon_4, \alpha_3=\epsilon_4, \alpha_4=\frac{1}{2}(\epsilon_1-\epsilon_2-\epsilon_3-\epsilon_4)\Big\},$$
and in the case of $G_2$ we have
$$\Big\{\alpha_1=\frac{1}{\sqrt{3}}(-2\epsilon_1+\epsilon_2+\epsilon_3) , \alpha_2={\frac{1}{\sqrt{3}}(\epsilon_1-\epsilon_2)}\Big\},$$ 
with the following Dynkin diagram 
\begin{center}\begin{tikzpicture}[start chain]
\dnodenj{1}
\dnodenj{2}
\path (chain-1) -- node {\(\Rrightarrow\)} (chain-2);
\end{tikzpicture} \ .
\end{center}
For every $\alpha \in R_{\pm}$, denote by $x_{\alpha}$ the generator of $\mathfrak{n}_{\pm}$. Denote by $\{\lambda_1, \ldots , \lambda_l\}$ the set of fundamental weights of $\mathfrak{g}$, where 
$$ \lambda_i=  \epsilon_{1}+\cdots +\epsilon_{i} \ \text{for} \  i \neq l, \ \text{and} \ \lambda_l= \frac{1}{2}(\epsilon_{1}+\cdots +\epsilon_{l})  \ \text{in the case of} \ B_l,$$ 
$$  \lambda_i= \frac{1}{\sqrt{2}}(\epsilon_{1}+\cdots +\epsilon_{l-i+1})  \ \text{in the case of} \ C_l,$$ 
$$ \lambda_1=  \epsilon_{1}+ \epsilon_{2}, \lambda_2=2\epsilon_{1}+\epsilon_{2}+\epsilon_{3}, \lambda_{3}=\frac{1}{2}(3\epsilon_{1}+\epsilon_{2}+\epsilon_{3}+\epsilon_{4}),\lambda_4= \epsilon_1  \ \text{in the case of} \ F_4, $$ 
$$ \lambda_1=\frac{1}{\sqrt{3}}(-\epsilon_1-\epsilon_2+2\epsilon_3) , \lambda_2={\frac{1}{\sqrt{3}}(-\epsilon_2+\epsilon_3)}  \ \text{in the case of} \ G_2.$$
We identify $\mathfrak{h}$ with $\mathfrak{h}^{\ast}$ using form $\left< \cdot, \cdot \right>$. Thus, the fundamental weights are viewed as elements of $\mathfrak{h}$ (cf. \cite{H}).

The affine Kac-Moody Lie algebra $ \widetilde{\mathfrak{g}}$ associated with $\mathfrak{g}$ is infinite-dimensional vector space
$$ \widetilde{\mathfrak{g}}=\mathfrak{g}\otimes \mathbb{C}[t,t^{-1}]\oplus \mathbb{C}c\oplus \mathbb{C}d,$$
where $c$ denotes the canonical central element and $d$ denotes the degree operator, equipped with the commutation relations 
$$
\left[x(m),y(n)\right]= \left[x, y\right](m+n)+ \left\langle x, y \right\rangle m \delta_{m+n\,0}\, c, 
$$
$$\left[d,x(m)\right]=mx(m)\ \text{and} \left[d,c\right]=0,$$
for all $x,y \in \mathfrak{g}$, $m,n \in \mathbb{Z}$ (cf. \cite{K}).  The generating functions for elements $x(m)=x\otimes t^n$ of the affine algebra are defined by 
\beq \nonumber
x(z)=\sum_{m \in \mathbb{Z}}x(m)z^{-m-1}.
\eeq
Denote by $\{\alpha_0, \alpha_1, \ldots , \alpha_l\}$ the set of simple roots, and by $\{\Lambda_0, \Lambda_1, \ldots , \Lambda_l\}$ the set of fundamental weights of $ \widetilde{\mathfrak{g}}$. 

We consider rectangular weights, i.e. the  highest weights of the form 
\beq \label{s1}
\Lambda=k_0\Lambda_0+k_j\Lambda_j,
\eeq
where $k_0, k_j \in \mathbb{Z}_+$ and $\Lambda_j$ denotes the fundamental weight such that $\left< \Lambda_j, c\right>=1$.  When $ \widetilde{\mathfrak{g}}$ is of type $B_l^{(1)}$ we have $j=1,l$, in the case of $C_l^{(1)}$ $j=1, \ldots, l$, in the case of $F_4^{(1)}$ $j$ is equal to 4 and in the case of $G_2^{(1)}$ $j=2$ (cf. \cite{K}). Denote by $L(\Lambda)$ the standard (i.e. integrable highest weight) $ \widetilde{\mathfrak{g}}$-module with a highest weight as in (\ref{s1}). With $k=\Lambda(c)$ denote the level of $ \widetilde{\mathfrak{g}}$-module $L(\Lambda)$, $k = k_0 +k_j$.

For every simple root $\alpha_i$, $1 \leq i \leq l$ denote by $\mathfrak{sl}_2(\alpha_i) \subset \mathfrak{g}$ a
subalgebra generated by $x_{\alpha_i}$ and $x_{-\alpha_i}$, and let 
$\widetilde{\mathfrak{sl}}_2(\alpha_i) = \mathfrak{sl}_2(\alpha_i) \otimes \mathbb C[t, t\sp{-1}]
\oplus \mathbb{C}c_{\alpha_i} \oplus \mathbb{C}d \subset \widetilde{\mathfrak{g}}$ 
be the corresponding affine Lie algebra of type $A^{(1)}_1$ with the canonical central element 
\beq\nonumber c_{\alpha_i} = \frac{2c}{\left<\alpha_i, \alpha_i \right>}.
\eeq
The restriction of $L(\Lambda)$ to $\widetilde{\mathfrak{sl}}_2(\alpha)$ is standard module of level
\beq\nonumber k_{\alpha_i} = \frac{2k}{\left<\alpha_i, \alpha_i \right>}.
\eeq
For later use we introduce the following notation
\beq \label{jt}
j_t=\left\{ \begin{array}{ccrcccl} 0  & \text{for} &1 &  \leq&  t& \leq& \nu_j k_0+(\nu_j-1)k_j, \ t > k_{\alpha_j}\\ j  & \text{for} & \nu_j k_0+(\nu_j-1)k_j+1 &\leq & t&\leq &k_{\alpha_j}\end{array} \right. ,
\eeq
where $\nu_j$ denotes $\frac{2}{\left<\alpha_j, \alpha_j\right>}$.

For each fundamental $ \widetilde{\mathfrak{g}}$-module $L(\Lambda_j)$ fix a highest weight vector $v_{\Lambda_j}$. By complete reducibility of tensor products
of standard modules, for level $k > 1$ we have
\beq \nonumber
L(\Lambda) \subset L(\Lambda_j)^{\otimes k_j}\otimes  L(\Lambda_0)^{\otimes k_0} ,
\eeq
with a highest weight vector
\beq \nonumber
v_{\Lambda} =  v_{\Lambda_j}^{\otimes k_j} \otimes  v_{\Lambda_0}^{\otimes k_0}.
\eeq

Consider $ \widetilde{\mathfrak{g}}$-subalgebra
\beq \nonumber
\widetilde{\mathfrak{n}}_{+}=\mathfrak{n}_{+} \otimes \mathbb{C}[t,t^{-1}].
\eeq
{\em The principal subspace} $W_{L(\Lambda)}$ of $L(\Lambda)$ is defined as
\beq \nonumber
W_{L(\Lambda)}=U\left(\widetilde{\mathfrak{n}}_{+}\right)  v_{\Lambda}, 
\eeq
(cf. \cite{FS}). 

This space is generated by operators from
\begin{equation*}
U = U(\widetilde{\mathfrak{n}}_{\alpha_l})\cdots U(\widetilde{\mathfrak{n}}_{\alpha_1}),\end{equation*}
which act on the highest weight vector $v_{\Lambda}$ (see Lemma 3.1 in \cite{G1} and also \cite{Bu1,Bu2,Bu3,BK}), where  \beq \nonumber
\widetilde{\mathfrak{n}}_{\alpha_i}=\mathbb{C}x_{\alpha_i} \otimes \mathbb{C}[t,t^{-1}], \ \ 1 \leq i \leq l.
\eeq

\section{Quasi-particle bases}
In this section, we first recall the notion and some basic facts about quasi-particles from \cite{G1,Bu1,Bu2,Bu3, BK}. Then we determine the spanning set of the principal subspace $W_{L(\Lambda)}$.

Recall that the simple vertex operator algebra $L(k\Lambda_0)$ associated with the integrable highest weight module of $\widetilde{\mathfrak{g}}$ with level $k$ is generated by $x(-1)v_{k\Lambda_0}$ for $x \in \mathfrak{g}$ such that
$$
Y(x(-1)v_{k\Lambda_0}, z)=x(z),
$$
where $v_{k\Lambda_0}$ is the vacuum vector, (cf. \cite{FLM,LL}). Moreover, the level $k$ standard $\widetilde{\mathfrak{g}}$-modules are modules for this vertex operator algebra. 

We will consider the vertex operators 
\beq\label{vopquasi}
x_{r\alpha_i}(z)=Y(x_{\alpha_i}(-1)^r v_{k\Lambda_0}, z)=\sum_{m\in\mathbb{Z}} x_{r\alpha_i}(m) z^{-m-r}=\underbrace{x_{\alpha_i}(z)\cdots x_{\alpha_i}(z)}_{r\text{ times}}
\eeq
associated with the vector $x_{\alpha_i}(-1)^r v_{k\Lambda_0} \in L(k\Lambda_0)$. Following \cite{G1}, for a fixed positive integer $r$ and a fixed integer $m$ define the  {\em quasi-particle of color $i$, charge $r$ and  energy $-m$} as the coefficient  $x_{r\alpha_i}(m)$ of \eqref{vopquasi}.

Note that charges of quasi-particles $x_{r\alpha_i}(z)$ in our quasi-particle basis monomial will be less or equal to $k_{\alpha_i}$, since 
\beq\label{rel1}
x_{(k_{\alpha_i}+1)\alpha_i}(z)=0
\eeq
on $L(\Lambda)$ (see \cite{LL}, \cite{LP}, \cite{MP}). Also, from the definition (\ref{vopquasi}) follows that
\beq\label{rel2}
x_{r\alpha_i}(z)v_{k\Lambda_0} \in W_{L(\Lambda)}[[z]].
\eeq

In the case of $L(\Lambda_j)$ we have the following relations
\begin{lem}\label{pocuv1}
 In the case of affine Lie algebras $ \widetilde{\mathfrak{g}}$ of type $B_l^{(1)}$ and  $C_l^{(1)}$ on $L(\Lambda_1)$, we have
\begin{eqnarray}\label{inBC1}
x_{ \alpha_1}(-1)v_{\Lambda_1}&=&0,\\
\label{inBC2}
x_{ \alpha_1}(-2)v_{\Lambda_1}&\neq&0,\\
\label{inBC3}
x_{ \alpha_i}(-1)v_{\Lambda_1}&\neq &0, \ \text{for} \  i \neq 1,\\
\label{inBC4}
x_{ 2\alpha_i}(-2)v_{\Lambda_1}&\neq&0, \ \text{if} \  \left< \alpha_i, \alpha_i\right>=1.
\end{eqnarray}
\end{lem}
\begin{proof}
Let us first assume that $ \widetilde{\mathfrak{g}}$ is of type $B_l^{(1)}$. Since the restriction of $L(\Lambda_1)$ to $\widetilde{\mathfrak{sl}}_2(\alpha_1)$ is a level one module and since we have $\left< \lambda_1, \alpha_1\right>=1$, it follows that $\widetilde{\mathfrak{sl}}_2(\alpha_1)v_{\Lambda_1}$ is a standard $A^{(1)}_1$-module $L(\Lambda_1)$. This gives us 
\begin{eqnarray}\nonumber
x_{ \alpha_1}(-1)v_{\Lambda_1}&=&0\\
\nonumber
x_{ \alpha_1}(-2)v_{\Lambda_1}&\neq&0.\end{eqnarray}
The restriction of $L(\Lambda_1)$ to $\widetilde{\mathfrak{sl}}_2(\alpha_i)$, where $i \neq 1$, is a level one module with trivial $\mathfrak{sl}_2(\alpha_i)$-module on the top, and therefore it is a standard $A^{(1)}_1$-module $L(\Lambda_0)$. From this follows (\ref{inBC3}). Relation (\ref{inBC4}) follows from the fact that the restriction of $L(\Lambda_1)$ to $\widetilde{\mathfrak{sl}}_2(\alpha_l)$ is a level two module with trivial $\mathfrak{sl}_2(\alpha_i)$-module on the top, and therefore it is a standard $A^{(1)}_1$-module $L(2\Lambda_0)$ (cf. \cite{K}). 

In a similar way it can be verified that the claims of the lemma hold for the case of a $C_l^{(1)}$-module $L(\Lambda_1)$.
\end{proof}

\begin{lem}\label{pocuv2}
 In the case of affine Lie algebras $ \widetilde{\mathfrak{g}}$ of type $B_l^{(1)}$,  $C_l^{(1)}$, $F_4^{(1)}$,  $G_2^{(1)}$ on $L(\Lambda_j)$, where $j \neq 1$, we have
\begin{eqnarray}\label{inBCFG5}
x_{ \alpha_i}(-1)v_{\Lambda_j}&\neq &0, \ \text{for} \ 1 \leq  i \leq l,\\
\label{inBCFG6}
x_{ 2\alpha_i}(-2)v_{\Lambda_j}&\neq&0, \ \text{for} \ i \neq j \ \text{and} \ \left< \alpha_i, \alpha_i\right>=1,\\
\label{inBCFG7}
x_{ 2\alpha_j}(-2)v_{\Lambda_j}&= &0, \ \text{for} \ \left< \alpha_j, \alpha_j\right>=1,\\
\label{inBCFG8}
x_{ 2\alpha_j}(-3)v_{\Lambda_j}&\neq&0, \ \text{for} \  \left< \alpha_j, \alpha_j\right>=1,\\
\label{inBCFG9nak} 
x_{ 2\alpha_j}(-2)v_{\Lambda_j}&\neq &0, \ \text{for} \ \left< \alpha_j, \alpha_j\right>=\frac{2}{3},\\
\label{inBCFG9} 
x_{ 3\alpha_j}(-3)v_{\Lambda_j}&= &0, \ \text{for} \ \left< \alpha_j, \alpha_j\right>=\frac{2}{3},\\
\label{inBCFG10}
x_{ 3\alpha_j}(-4)v_{\Lambda_j}&\neq &0, \ \text{for} \ \left< \alpha_j, \alpha_j\right>=\frac{2}{3}.\end{eqnarray}
\end{lem}
\begin{proof}
Let $ \widetilde{\mathfrak{g}}$ be of type $B_l^{(1)}$. The restriction of $L(\Lambda_l)$ to $\widetilde{\mathfrak{sl}}_2(\alpha_i)$, where $i \neq l$, is a standard $A^{(1)}_1$-module $L(\Lambda_0)$. Therefore, 
\beq\nonumber
x_{ \alpha_i}(-1)v_{\Lambda_l} \neq  0.\eeq
On the other hand the restriction of $L(\Lambda_l)$ to $\widetilde{\mathfrak{sl}}_2(\alpha_l)$ is a level two module with two dimensional $\mathfrak{sl}_2(\alpha_l)$-module on the top, so it must be a standard $A^{(1)}_1$-module $L(\Lambda_0+\Lambda_1)$. From this follows \begin{eqnarray}\nonumber
x_{ \alpha_l}(-1)v_{\Lambda_l}&\neq &0, \\
\nonumber
x_{ 2\alpha_l}(-2)v_{\Lambda_l}&=&0,\\
\nonumber
x_{ 2\alpha_l}(-3)v_{\Lambda_l}&\neq &0.
\end{eqnarray}
In a similar way it can be verified that relations (\ref{inBCFG5}), (\ref{inBCFG7}) and (\ref{inBCFG8}) hold for the case of a $C_l^{(1)}$-module $L(\Lambda_j)$, where $j=2, \ldots ,l$ and $F_4^{(1)}$-module $L(\Lambda_4)$. In the case of $C_l^{(1)}$-module $L(\Lambda_j)$ and $F_4^{(1)}$-module $L(\Lambda_4)$ we also have relation (\ref{inBCFG6}), which follows from the fact that the restriction of $L(\Lambda_j)$ to $\widetilde{\mathfrak{sl}}_2(\alpha_i)$, where $i \neq j, 1$ (or $i = 3$ in the case of $F_4^{(1)}$) is a level two module with trivial $\mathfrak{sl}_2(\alpha_i)$-module on the top, and therefore it must be a standard $A^{(1)}_1$-module $L(2\Lambda_0)$. When $ \widetilde{\mathfrak{g}}$ is of type $G_2^{(1)}$ we have relations (\ref{inBCFG5}), \eqref{inBCFG9nak},  (\ref{inBCFG9}) and (\ref{inBCFG10}) on $L(\Lambda_2)$. These relations are a consequence of the fact that the restriction of $L(\Lambda_2)$ to $\widetilde{\mathfrak{sl}}_2(\alpha_1)$ is a level one standard $L(\Lambda_0)$-module and the restriction of $L(\Lambda_2)$ to $\widetilde{\mathfrak{sl}}_2(\alpha_1)$ is $A^{(1)}_1$-module $L(2\Lambda_0+\Lambda_1)$. 
\end{proof}
From the last two lemmas we have
\beq\label{rel3}
x_{r\alpha_i}(z) v_{\Lambda_j}^{\otimes k_j} \otimes  v_{\Lambda_0}^{\otimes k_0} \in z^{\sum_{t=1}^{r}\delta_{i, j_t}}W_{L(\Lambda)}[[z]].
\eeq

Our quasi-particle basis monomial will be of the form
\beq\label{briefly}
b=b_{\alpha_l}\,\cdots \,b_{\alpha_2}b_{\alpha_1}, 
\eeq
where 
\beq\nonumber
b_{\alpha_i}=x_{n_{r_{i}^{(1)},i}\alpha_{i}}(m_{r_{i}^{(1)},i})\ldots x_{n_{1,i}\alpha_{i}}(m_{1,i}),\eeq
\beq\nonumber
n_{r_{i}^{(1)},i} \leq\cdots \leq n_{1,i}\leq k_{\alpha_i} \quad\text{and}\quad  r_{i}^{(1)} \geq r_{i}^{(2)}\geq \ldots \geq r_{i}^{(k_{\alpha_i})}\text{ for } i=1,\ldots ,l.
\eeq 
Here $n_{p,i}$ and $r_i^{(t)} $ represent parts of a conjugate pair of partitons $\Cc_i=(n_{r_{i}^{(1)},i},\ldots , n_{1,i})$ and $\Dc_i=(r_{i}^{(1)}, r_{i}^{(2)}, \ldots , r_{i}^{(s)})$ of some fixed $n_i$. Following \cite{G1} we call $\Cc_i$ a charge-type of a monomial $b_{\alpha_i}$, $\Dc_i$ a dual-charge-type and $n_i$ a color-type of a monomial $b_{\alpha_i}$. We can visualize charge-type and dual-charge type of monomial $b_{\alpha_i}$ using graphic presentation, as in the following example.

\begin{example}\label{ex1}
For monomial $$x_{\alpha_{i}}(m_{4,i})x_{2\alpha_{i}}(m_{3,i})x_{4\alpha_{i}}(m_{2,i})x_{4\alpha_{i}}(m_{1,i})$$
of color-type $n_i=11$, of charge-type $\Cc_i=(1,2,4,4)$ and dual-charge-type $\Dc_i=(4,3,2,2)$
 we have the graphic presentation as given in Figure \ref{figure1}, where each quasi-particle of charge $r$ is presented by a column of height $r$. Number of boxes in every row represents part of a dual-charge-type $\Dc_i$.

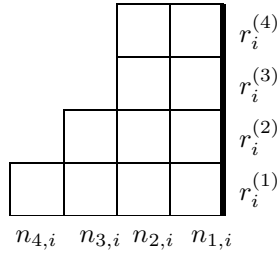
\begin{figure}[h!tb]
\setlength{\unitlength}{7mm}
\begin{picture}(5,4)
\linethickness{0.1mm}
\multiput(0,0)(1,0){1}%
{\line(0,1){1}}
\multiput(0,0)(1,0){4}%
{\line(1,0){1}}
\multiput(0,1)(1,0){4}%
{\line(1,0){1}}
\multiput(1,2)(1,0){3}%
{\line(1,0){1}}
\multiput(1,0)(1,0){2}%
{\line(0,1){2}}
\multiput(3,3)(1,0){1}%
{\line(1,0){1}}
\multiput(2,2)(1,0){2}%
{\line(0,1){2}}
\multiput(2,3)(1,0){1}%
{\line(1,0){1}}
\multiput(2,4)(1,0){1}%
{\line(1,0){1}}
\multiput(3,0)(1,0){2}%
{\line(0,1){4}}
\multiput(3,4)(1,0){1}%
{\line(1,0){1}}
\linethickness{0.75mm}
\multiput(4,0)(1,0){1}%
{\line(0,1){4}}
\put(4.3,0.3){\footnotesize{$r_i^{(1)}$}}
\put(4.3,1.3){\footnotesize{$r_i^{(2)}$}}
\put(4.3,2.3){\footnotesize{$r_i^{(3)}$}}
\put(4.3,3.3){\footnotesize{$r_i^{(4)}$}}
\put(3.4, -0.5){\footnotesize{$n_{1,i}$}}
\put(2.3,-0.5){\footnotesize{$n_{2,i}$}}
\put(1.3, -0.5){\footnotesize{$n_{3,i}$}}
\put(0.1, -0.5){\footnotesize{$n_{4,i}$}}
\end{picture}
\caption{Graphic presentation}
\label{figure1}
\end{figure}
\end{example}

Analogously to the situation with a single color, we define the {\em charge-type} $\Cc$, the {\em dual-charge-type} $\Dc$ of $b$ in (\ref{briefly}) by
\begin{align}
&\Cc=\left( \Cc_l; \,
\ldots ;\,
 \Cc_1\right),\label{charge-type}\\
&\Dc=\left(\Dc_l; \,
\ldots ;\,
 \Dc_1\right)\label{dual-charge-type},
\end{align}
and the {\em color-type} of $b$ as the l-tuple $(n_{l},\ldots,n_{1})$ where $n_i$ denotes the color-type of a monomial $b_{\alpha_i}$. Moreover, by
$$\Ec=\left( m_{r_{l}^{(1)},l},\ldots , m_{1,l}; \,
\ldots ;\,
 m_{r_{1}^{(1)},1},\ldots  , m_{1,1}\right)$$ 
we denote the {\em energy-type} of $b$.

Now, let $b, \overline{b}$ be any two quasi-particle monomials of the same color-type, expressed as in (\ref{briefly}). Denote their charge-types and energy-types by $\Cc,\overline{\Cc}$  and  $\Ec,\overline{\Ec}$ respectively. We define the linear order among quasi-particle monomials of the same color-type by 
\beq\label{order1}
b< \overline{b}\qquad \text{if}\qquad \Cc<\overline{\Cc}\quad\text{or}\quad \Cc=\overline{\Cc} \text{ and } \Ec <\overline{\Ec},
\eeq
 where for (finite) sequences of integers we define:
$$(x_p,\ldots ,x_1)< (y_r,\ldots ,y_1)$$
if there exists $s$ such that 
\beq\label{order2}x_1=y_1,\,\ldots,\,x_{s-1}=y_{s-1}\ \text{and} \ s=p+1\leqslant r\quad \text{or}\quad x_s<y_s.
\eeq

From (\ref{rel2}) and (\ref{rel3}) follows that energies in the expression obtained by applying (\ref{briefly}) on the highest weight vector comply the following difference condition
\beq\label{rel4}
m_{p,i} \leq -n_{p,i} -\sum_{t=1}^{n_{p,i}}\delta_{i, j_t}, \ \text{for} \ 1\leq p \leq r_i^{(1)}.
\eeq
We strengthen this inequality by using relations among quasi-particles. 

The interactions among quasi-particles of different colors \cite[Lemma 2.3.2]{Bu1}, \cite[Lemma 4.3, Lemma 5.3]{Bu2}, \cite[Lemma 3.2]{Bu3}, \cite[Lemma 5.3]{BK} we summarize as follows. 
\begin{lem}\label{rel5}
For quasi-particles of fixed charges $n_{i-1}$ and $n_i $ on $W_{L(\Lambda)}$ we have
\begin{align}
(z_{1}-z_{2})^{ M_{i}}x_{n_{i}\alpha_{i}}(z_{2})
x_{n_{i-1}\alpha_{i-1}}(z_{1})
=(z_{1}-z_{2})^{M_{i} }
x_{n_{i-1}\alpha_{i-1}}(z_{1})x_{n_{i}\alpha_{i}}(z_{2}), 
\end{align}
where
$
M_{i}=\min \textstyle\left\{\textstyle\frac{\nu_{\alpha_i}}{\nu_{\alpha_{i-1}} }n_{i-1},n_i\right\}$.
\end{lem}
The interaction among quasi-particles of the same color is described by the following assertion \cite[Lemma 3.3]{F}, \cite[Lemma 4.4]{JP}, \cite[(3.18)--(3.23)]{G1}.
\begin{lem}\label{rel6}
For fixed charges $n_1,n_2$ such that $n_2\leqslant n_1$ and fixed integer $M$ such that $m_1+m_2=M$ the monomials  
$$
x_{n_2\alpha_i}(m_2)x_{n_1\alpha_i}(m_1) ,\, x_{n_2\alpha_i}(m_2-1)x_{n_1\alpha_i}(m_1 +1) ,\,\ldots \,  ,
 x_{n_2\alpha_i}(m_2-2n_2+1)x_{n_1\alpha_i}(m_1+2n_2-1)
$$
of  operators on $W_{L(\Lambda)}$
can be expressed
as a linear combination of monomials
$$
x_{n_2\alpha_i}(j_2)x_{n_1\alpha_i}(j_1) \quad \text{such that} \quad j_2 \leqslant m_2- 2n_2,\quad  j_1\geqslant m_1+2n_2 \ \text{and} \ j_1 +j_2=M
$$
and monomials which contain a quasi-particle of color $i$ and charge $n_1+1$. 
Moreover, for $n_2 = n_1$ the monomials 
$$
x_{n_2 \alpha_i}(m_2)x_{n_2 \alpha_i}(m_1)\quad\text{with} \ \ m_1-2n_2< m_2 \leqslant   m_1
$$
can be expressed
as a linear combination of monomials
$$
x_{n_2\alpha_i}(j_2)x_{n_2\alpha_i}(j_1)  \quad\text{such that} \quad j_2\leqslant j_1-2n_2 \ \text{and} \  j_1 +j_2=M$$
and monomials which contain a quasi-particle of color $i$ and charge $n_2+1$.
\end{lem}
Denote by $B_W$ the set of all quasi-particle monomials of the form as in \eqref{briefly} which satisfy the following difference conditions
\begin{align}\label{rel7}
m_{p,i} &\leq -n_{p,i} +\sum_{q=1}^{r_{i-1}^{(1)}}\min \textstyle\left\{\textstyle\frac{\nu_{\alpha_i}}{\nu_{\alpha_{i-1}} }n_{i-1},n_i\right\}-2(p-1)n_{p,i}-\sum_{t=1}^{n_{p,i}}\delta_{i, j_t}, \ \text{for} \ 1\leq p \leq r_i^{(1)},\\
\label{rel8}
m_{p+1,i}& \leq m_{p,i} -2n_{p,i}, \ \text{for} \ n_{p+1,i}=n_{p,i}, \ \ 1\leq p \leq r_i^{(1)}-1,
\end{align}
where $r_0^{(1)}=0$ and $j_t$ is as in \eqref{jt}. We have
\begin{thm}\label{t1}
The set $\mathcal{B}_W=\{bv_{\Lambda} : b \in B_W \}$ forms a basis of the principal space $W_{L(\Lambda)}$.
\end{thm}

The proof that $\mathcal{B}_W$ is the spanning set goes as in \cite{G1}, by using induction on the charge-type and the total energy of quasi-particle monomials. It remains to prove the linear independence of the spanning set.

\section{Proof of linear independence}
In the proof of linear independence of the set $\mathcal{B}_W$ we will employ operators defined on level one standard modules $L(\Lambda_j)$. The projection $\pi_{\Dc}$, which generalizes the projection introduced in \cite{G1} (see also \cite{Bu1, Bu2, Bu3, BK}), enables us to use these operators in the case of higher levels. In Section \ref{ss3.1} we will recall the main properties of $\pi_{\Dc}$. In Section \ref{ss3.2} we will introduce coefficients of intertwining operators among level one modules and in Section \ref{ss3.3} we will recall important properties of Weyl group translation operators. Finally in Section \ref{ss3.4} we prove linear independence of the set $\mathcal{B}_W$.

\subsection{Projection $\pi_{\Dc}$}\label{ss3.1}
For a dual-charge-type $\Dc$ of monomial \eqref{briefly} denote by $\pi_{\Dc}$ the projection of $W_{L(\Lambda)}$ on the vector space
$$ {W_{L(\Lambda_{j^k})}}_{(\mu^{(k)}_{l};\ldots;\mu_{1}^{(k)})}\otimes \cdots \otimes  {W_{L(\Lambda_{j^1})}}_{(\mu^{(1)}_{l};\ldots;\mu_{1}^{(1)})}\subset   W_{ L(\Lambda_j)}^{\otimes k_j}\otimes W_{L(\Lambda_0)}^{\otimes k_0} \subset   L(\Lambda_j)^{\otimes k_j} \otimes L(\Lambda_0)^{\otimes k_0},
$$
where $j^t \in \{0,j\}$, $1 \leq t \leq k$, ${W_{L(\Lambda_{j^t})}}_{(\mu^{(t)}_{l};\ldots;\mu_{1}^{(t)})}$
denotes the $\mathfrak{h}$-weight subspace of the level one principal subspace $W_{ L(\Lambda_{j^t})}$ of weight $\mu^{(t)}_{l}\alpha_l+ \cdots+ \mu_{1}^{(t)}\alpha_1 \in Q$ with
\beq\label{proj0}
\mu^{(t)}_{i}=\sum_{p=0}^{\nu_i-1} r^{(\nu_it-p)}_{i} \ \text{for} \  1 \leq t \leq k.
\eeq
With the same symbol we denote the generalization of the projection $\pi_{\Dc}$ to the space of formal series with coefficients in $W_{ L(\Lambda_j)}^{\otimes k_j}\otimes W_{L(\Lambda_0)}^{\otimes k_0}$. Let 
\beq\label{eq:p1}
x_{n_{r_{l}^{(1)},l}\alpha_{l}}(z_{r_{l}^{(1)},l}) \cdots     x_{n_{1,l}\alpha_{l}}(z_{1,l})\cdots x_{n_{r_{1}^{(1)},1}\alpha_{1}}(z_{r_{1}^{(1)},1})\cdots  x_{n_{1,1}\alpha_{1}}(z_{1,1})\ v_{ \Lambda}\eeq
be the generating function of the monomial \eqref{briefly}, which acts on the highest weight vector $v_{ \Lambda}$. From relations (\ref{rel1}) follows that the projection of \eqref{eq:p1} is:
\begin{equation}\label{eq:p2}
\pi_{\Dc} \left( x_{n_{r_{l}^{(1)},l}\alpha_{l}}(z_{r_{l}^{(1)},l})\cdots  x_{n_{1,1}\alpha_{1}}(z_{1,1}) \ v_{ \Lambda}\right)
\end{equation}
$$=\text{C}  x_{n_{r^{(\nu_l(k-1)+1)}_{l},l}^{(k)}\alpha_{l}}(z_{r_{l}^{(\nu_l(k-1)+1)},l})\cdots  x_{n_{r^{(\nu_lk)}_{l},l}^{(k)}\alpha_{l}}(z_{r_{l}^{(\nu_lk)},l})\cdots   x_{n_{1,l}^{(k)}\alpha_{l}}(z_{1,l}) \cdots$$
$$ \ \ \ \ \ \ \ \ \ \ \ \ \ \ \ \   \cdots  x_{n_{r^{(\nu_1(k-1)+1)}_{1},1}^{(k)}\alpha_{1}}(z_{r_{1}^{(\nu_1(k-1)+1)},1})\cdots  x_{n_{r^{(\nu_1k)}_{1},1}^{(k)}\alpha_{1}}(z_{r_{1}^{(\nu_1k)},1})\cdots    x_{n{_{1,1}^{(k)}\alpha_{1}}}(z_{1,1}) \ v_{ \Lambda_{j^k}},$$
$$ \ \ \ \ \ \ \ \ \ \ \ \ \ \ \ \ \otimes \cdots \otimes$$
$$
\otimes x_{n_{r^{(1)}_{l},l}^{(1)}\alpha_{l}}(z_{r_{l}^{(11)},l})\cdots  x_{n_{r^{(\nu_lk)}_{l},l}^{(1)}\alpha_{l}}(z_{r_{l}^{(\nu_lk)},l})\cdots   x_{n_{1,l}^{(1)}\alpha_{l}}(z_{1,l}) \cdots$$
$$ \ \ \ \ \ \ \ \ \ \ \ \ \ \ \ \   \cdots  x_{n_{r^{(1)}_{1},1}^{(1)}\alpha_{1}}(z_{r_{1}^{(1)},1})\cdots  x_{n_{r^{(\nu_1k)}_{1},1}^{(1)}\alpha_{1}}(z_{r_{1}^{(\nu_1k)},1})\cdots    x_{n{_{1,1}^{(1)}\alpha_{1}}}(z_{1,1}) \ v_{ \Lambda_{j^1}},$$
where $\text{C} \in \mathbb{C}^{*}$, and where
$$
0 \leq  n^{(t)}_{p,i}\leq  \nu_i, \  \ n_{p,i}=\sum_{t=1}^k n^{(t)}_{p,i}, \ \ \text{for every}   \  1 \leq p \leq r_{i}^{(1)}.$$
For fixed color $i$ the projection $\pi_{\Dc}$ places at most $ \nu_i$ generating functions $x_{\alpha_i}(z_{p,i})$ on
each tensor factor $v_{\Lambda_{j^t}}$, $1 \leq t \leq k$. This property of $\pi_{\Dc}$ is demonstrated in the following example for the case of affine Lie algebra $\widetilde{\mathfrak{g}}$ of type $G_2^{(1)}$.
\begin{example} Consider the formal power series
\beq\label{ex2}x_{\alpha_{2}}(z_{3,2}) x_{3\alpha_{2}}(z_{2,2})x_{4\alpha_{2}}(z_{1,2}) x_{\alpha_{1}}(z_{2,1})x_{2\alpha_{1}}(z_{1,1}) v_{\Lambda}\eeq
with coefficients in the principal subspace $W_{L(\Lambda_0+\Lambda_2)}$ of level 2 standard module $L(\Lambda_0+\Lambda_2)$ of affine Lie algebra $\widetilde{\mathfrak{g}}$ of type $G_2^{(1)}$. The projection $\pi_{\Dc}$ of \eqref{ex2}, where $\Dc=(3,2,2,1;2,1)$, onto 
$${W_{L(\Lambda_2)}}_{(1; 1)}\otimes {W_{L(\Lambda_0)}}_{(7; 2)}$$ 
is
\begin{align}\label{ex2proj}
& Cx_{\alpha_{2}}(z_{1,2})x_{2\alpha_{1}}(z_{1,1})v_{\Lambda_2}\otimes x_{\alpha_{2}}(z_{3,2}) x_{3\alpha_{2}}(z_{2,2})x_{3\alpha_{2}}(z_{1,2}) x_{\alpha_{1}}(z_{2,1})x_{\alpha_{1}}(z_{1,1}) v_{\Lambda_0},\end{align}
($\text{C} \in \mathbb{C}^{*}$). Graphically, the image of \eqref{ex2} can be represented as in Figure \ref{figure2}, where boxes in columns represent $n^{(t)}_{p,i}$.

\begin{figure}[htb!]
\centering
\setlength{\unitlength}{7mm}
\begin{picture}(8,4)
\linethickness{0.1mm}
\multiput(0,0)(1,0){6}%
{\line(0,1){1}}
\multiput(0,0)(1,0){5}%
{\line(1,0){1}}
\multiput(0,1)(1,0){1}%
{\line(1,0){5}}
\multiput(1,1)(1,0){3}%
{\line(0,1){2}}
\multiput(1,2)(1,0){1}%
{\line(1,0){2}}
\multiput(1,3)(1,0){1}%
{\line(1,0){2}}
\multiput(2,4)(1,0){1}%
{\line(1,0){1}}
\multiput(2,3)(1,0){1}%
{\line(0,1){1}}
\multiput(4,4)(1,0){1}%
{\line(1,0){1}}
\multiput(4,3)(1,0){1}%
{\line(0,1){1}}
\linethickness{0.75mm}
\multiput(0,3)(1,0){5}%
{\line(1,0){1}}
\multiput(3,0)(1,0){1}%
{\line(0,1){4}}
\multiput(5,0)(1,0){1}%
{\line(0,1){1}}
\multiput(5,3)(1,0){1}%
{\line(0,1){1}}
\put(6.5,1){\scriptsize{\footnotesize{$v_{ \Lambda_0 }$}}}
\put(6.5,3.3){\scriptsize{\footnotesize{$v_{ \Lambda_2}$}}}
\put(0.3, -0.5){\scriptsize{\footnotesize{$\alpha_2$}}}
\put(1.2, -0.5){\scriptsize{\footnotesize{$3\alpha_2$}}}
\put(2.2, -0.5){\scriptsize{\footnotesize{$4\alpha_2$}}}
\put(3.2, -0.5){\scriptsize{\footnotesize{$\alpha_1$}}}
\put(4.3, -0.5){\scriptsize{\footnotesize{$2\alpha_2$}}}
\end{picture}
\bigskip
\caption{ $\pi_\Dc\left(x_{\alpha_{2}}(z_{3,2}) x_{3\alpha_{2}}(z_{2,2})x_{4\alpha_{2}}(z_{1,2}) x_{\alpha_{1}}(z_{2,1})x_{2\alpha_{1}}(z_{1,1}) v_{\Lambda}\right)$}
\label{figure2}
\end{figure}
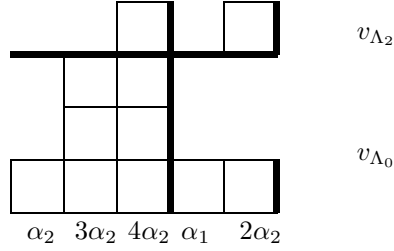

First note that we have $n_{1,1}^{(1)}=n_{1,1}^{(2)}=1$, since from the relation $x_{2\alpha_{1}}(z_{1,1})=0$ on $L(\Lambda_0+\Lambda_2)$ follows that with the projection $\pi_{\Dc}$ every factor $x_{\alpha_{1}}(z_{1,1})$ of the vertex operator $x_{2\alpha_{1}}(z_{1,1})$ is applied on the different tensor factor. With the projection the vertex operator $x_{\alpha_{1}}(z_{2,1})$ is applied only on the rightmost tensor factor, so $n_{2,1}^{(1)}=1$ and $n_{2,1}^{(2)}=0$. 
The relation $x_{4\alpha_{2}}(z_{1,2})=0$ on $L(\Lambda_0+\Lambda_2)$  implies that with the projection $\pi_{\Dc}$ three  vertex operators $x_{\alpha_{2}}(z_{1,2})$ are applied on the rightmost tensor factor and one vertex operator $x_{\alpha_{2}}(z_{1,3})$ is applied on the remaining tensor factor. From this follows that $n_{1,2}^{(1)}=3$ and $n_{1,2}^{(2)}=1$. With the projection the vertex operator $x_{3\alpha_{2}}(z_{2,2})$ is applied only on the rightmost tensor factor, so $n_{2,2}^{(1)}=3$ and $n_{2,2}^{(2)}=0$. Finally, the vertex operator $x_{\alpha_{2}}(z_{3,2})$ is applied on the rightmost tensor factor, therefore we have $n_{3,2}^{(1)}=1$ and $n_{3,2}^{(2)}=0$. 
\end{example}

\subsection{Coefficients of level $1$ intertwining operators and simple current maps}\label{ss3.2}
First let $ \widetilde{\mathfrak{g}}$ be of type $B_l^{(1)}$ or of type $C_l^{(1)}$. Denote by $I_1(\cdot, z)$ the intertwining operator of type $\binom{L(\Lambda_j)}{L(\Lambda_j)  \,\,  L(\Lambda_0)}$, defined by
\beq\label{sc0}
I_1(w_j, z)w_0=\text{exp}(zL(-1))Y(w_0,-z)w_j,
\eeq
where $w_j \in L(\Lambda_j)$ and $w_0 \in L(\Lambda_0)$ (cf. \cite{FHL}). Following  \cite{Bu1, Bu2} denote by $A_{\lambda_1}$ the constant term of the intertwining operator $I_1(v_{\Lambda_1}, z)$. This coefficient commutes with the action of quasi-particles (cf.  \cite{Bu1, Bu2}).

Fix $\lambda_1 \in \mathfrak{h}$ as in Section \ref{prelim}.  Following H. Li (cf. \cite{Li1, Li2, Li3}), for any $L(k\Lambda_0)$-module  $V$ we introduce the following notation $$(V^{(\lambda_1)}, Y_{\lambda_1}(\cdot, z))=(V, Y(\Delta(\lambda_1, z)\cdot, z)),$$
where 
$$\Delta(\lambda_1, z)=z^{\lambda_1}\text{exp}\left(\sum_{n\geq 1} \frac{\lambda_1(n)}{n}(-z)^{-n}\right).$$
By Proposition 2.6 in \cite{Li1} follows that $V^{(\lambda_1)}$ has a structure of a weak $L(k\Lambda_0)$-module. In particular, $L(k\Lambda_0)^{(\lambda_1)} \cong L(k\Lambda_1) $ is a simple current
$L(k\Lambda_0)$-module.

Following \cite{Bu1} and \cite{Bu2} denote by $e_{\lambda_1}$ simple current map, that is bijection
$$e_{\lambda_1}: L(\Lambda_j) \rightarrow  L(\Lambda_j)^{(\lambda_1)},$$
such that
\beq\label{sc1}
x_{\alpha}(m)e_{\lambda_{1}}=e_{\lambda_{1}}x_{\alpha}(m+\alpha(\lambda_{1})),
\eeq
for all $\alpha \in R$ and $m \in \mathbb{Z}$ and
\beq\label{sc2}
e_{\lambda_{1}}v_{\Lambda_0}=v_{\Lambda_1},
\eeq
(see \cite{DLM}, \cite{Li3},  or Remark 5.1 in \cite{P}). 

From \eqref{sc2} follows that the monomial vector $\pi_{\Dc}bv_{ \Lambda }\in \mathcal{B}_W$, where $W_{L(\Lambda)}$ is the principal subspace of the standard module $L(k_0\Lambda_0+k_1\Lambda_1)$ and  $b$ is of dual-charge type $\Dc$, equals 
\beq\label{sc10}
\pi_{\Dc}b(e_{\lambda_1}v_{\Lambda_0})^{\otimes k_1}\otimes v_{ \Lambda_0}^{\otimes k_0}.
\eeq
From \eqref{sc1} and from the definition of the projection $\pi_{\Dc}$ it follows that \eqref{sc10} is the coefficient of the variables 
\begin{eqnarray}\label{sc11}&z_{r_l^{(1)},l}^{-m_{r_l^{(1)},l}-n_{r_l^{(1)},l}} \cdots  z_{1,2}^{-m_{1,2}-n_{1,2}} z_{r_1^{(1)},1}^{-m_{r_1^{(1)},1}-n_{r_1^{(1)},1 }} \cdots\\
\nonumber
&\cdots  z_{r^{(k_0+1)}_{1}+1,1}^{-m_{r^{(k_0+1)}_{1}+1,1}-n_{r^{(k_0+1)}_{1}+1,1}}z_{r^{(k_0+1)}_{1},1}^{-m_{r^{(k_0+1)}_{1},1}-n_{r^{(k_0+1)}_{1},1}+(n_{r^{(k_0+1)}_{1},1}-k_0)}\cdots z_{1,1}^{-m_{1,1}-n_{1,1}+(n_{1,1}-k_0)}\end{eqnarray}
  in
\beq\label{sc12}\pi_{\Dc} \left( x_{n_{r_{l}^{(1)},l}\alpha_{l}}(z_{r_{l}^{(1)},l})\cdots   x_{n_{1,1}\alpha_{1}}(z_{1,1}) \ (e_{\lambda_1}v_{\Lambda_0})^{\otimes k_1}\otimes v_{ \Lambda_0}^{\otimes k_0}\right).\eeq
Denote this coefficient by $e_{\lambda_1}^{\otimes  k_1}\otimes 1^{\otimes k_0}\left(\pi_{\Dc}b^+v_{\Lambda_0} ^{\otimes k}\right)$, 
where 
\beq\label{sc13}
b^+ =b_{\alpha_l}\,\cdots \,b_{\alpha_2}b^+_{\alpha_1}\eeq
and
\begin{eqnarray}\label{sc14}b^+_{\alpha_1} &=x_{n_{r^{(1)}_{1},1}\alpha_{1}}(m_{r^{(1)}_{1},1})\cdots  x_{n_{r^{(k_0+1)}_{1}+1,1}\alpha_{1}}(m_{r^{(k_0+1)}_{1}+1,1})\\
\nonumber
& x_{n_{r^{(k_0+1)}_{1},1}\alpha_{1}}(m_{r^{(k_0+1)}_{1},1}+n_{r^{(k_0+1)}_{1},1}-k_0)\cdots  x_{n_{1,1}\alpha_{1}}(m_{1,1}+n_{1,1}-k_0).
\end{eqnarray}
From \eqref{rel7} and \eqref{rel8} follows that energies of quasi-particle monomial vectors $b^+v_{k\Lambda_0}$ satisfy difference conditions of energies of quasi-particle monomial vectors from the basis of principal subspace of the standard module $L(k\Lambda_0)$, (see also \cite{Bu2}).

For the case when $W_{L(\Lambda)}$ is the principal subspace of the standard module $L(k_0\Lambda_0+k_j\Lambda_j)$, $j \neq 1$, we will use the following lemma:  
\begin{lem}\label{lemoi}
In the case of affine Lie algebra $ \widetilde{\mathfrak{g}}$ of type $B_l^{(1)}$ 
\beq\label{sc3}
e_{\lambda_{1}}v_{\Lambda_l}=x_{\epsilon_1}(-1)v_{\Lambda_l}.
\eeq
In the case of affine Lie algebra $ \widetilde{\mathfrak{g}}$ of type $C_l^{(1)}$ and $j \neq 1$ 
\beq\label{sc4}
e_{\lambda_{1}}v_{\Lambda_j}=x_{\frac{1}{\sqrt{2}}(\epsilon_{l-j+1}+\epsilon_{j})}(-1)\cdots x_{\frac{1}{\sqrt{2}}(\epsilon_2+\epsilon_{l-1})}(-1)x_{\frac{1}{\sqrt{2}}(\epsilon_1+\epsilon_l)}(-1)v_{\Lambda_{l-j+2}}.
\eeq
\end{lem}
\begin{proof} Relations \eqref{sc3} and \eqref{sc4} are verified by arguing as in the proof in \cite[Lemma 5.3]{P}. $e_{\lambda_{1}}^{-1}x_{\epsilon_1}(-1)v_{\Lambda_l}$ is a weight vector with weight $\Lambda_l$, since we have
$$\alpha_1^\vee (0)e_{\lambda_{1}}^{-1}x_{\epsilon_1}(-1)v_{\Lambda_l}=e_{\lambda_{1}}^{-1}(\alpha_1^\vee(0)-c)x_{\epsilon_1}(-1)v_{\Lambda_l}=0,$$
$$\alpha_j^\vee (0) e_{\lambda_{1}}^{-1}x_{\epsilon_1}(-1)v_{\Lambda_l}=e_{\lambda_{1}}^{-1}\alpha_j^\vee (0) x_{\epsilon_1}(-1)v_{\Lambda_l}=0, \ \ \text{for} \ \ j \neq 1, l,$$
$$\alpha_l^\vee (0) e_{\lambda_{1}}^{-1}x_{\epsilon_1}(-1)v_{\Lambda_l}= e_{\lambda_{1}}^{-1}\alpha_l^\vee  (0) x_{\epsilon_1}(-1)v_{\Lambda_l}=e_{\lambda_{1}}^{-1}x_{\epsilon_1}(-1)v_{\Lambda_l},$$
$$x_{-\theta}(1) e_{\lambda_{1}}^{-1}x_{\epsilon_1}(-1)v_{\Lambda_l}=e_{\lambda_{1}}^{-1}x_{-\theta}(2) x_{\epsilon_1}(-1)v_{\Lambda_l}=0.$$
From $x_{\alpha_1}(-1) x_{\alpha_1}(-1)v_{\Lambda_l}=0$ follows that
$$x_{\epsilon_1}(0)x_{\alpha_1}(-1) x_{\alpha_1}(-1)v_{\Lambda_l}=2x_{\alpha_1}(-1) x_{\epsilon_1}(-1)v_{\Lambda_l}=0,$$ so we have 
$$x_{\alpha_1}(0)  e_{\lambda_{1}}^{-1}x_{\epsilon_1}(-1)v_{\Lambda_l}=e_{\lambda_{1}}^{-1}x_{\alpha_1}(-1) x_{\epsilon_1}(-1)v_{\Lambda_l}=0.$$
We also have
$$ x_{\alpha_j}(0) e_{\lambda_{1}}^{-1}x_{\epsilon_1}(-1)v_{\Lambda_l}= e_{\lambda_{1}}^{-1} x_{\alpha_j}(0) x_{\epsilon_1}(-1)v_{\Lambda_l}=e_{\lambda_{1}}^{-1}x_{\epsilon_1}(-1)x_{\alpha_j}(0)v_{\Lambda_l}=0,$$
for $j \neq1,l$, and
$$ x_{\alpha_l}(0) e_{\lambda_{1}}^{-1}x_{\epsilon_1}(-1)v_{\Lambda_l}= e_{\lambda_{1}}^{-1} x_{\alpha_l}(0) x_{\epsilon_1}(-1)v_{\Lambda_l}=e_{\lambda_{1}}^{-1}x_{\epsilon_1+\epsilon_l}(-1)v_{\Lambda_l}=0,$$
since the restriction of $L(\Lambda_l)$ on $\widetilde{\mathfrak{sl}}_2(\epsilon_1+\epsilon_l)$ is a level one standard $A_1^{(1)}$-module of highest weight $\Lambda_1$. Hence \eqref{sc3} holds and $L(\Lambda_l)^{(\lambda_1)} \cong L(\Lambda_l) $.

In a similar way we prove that $e_{\lambda_{1}}^{-1} u:=e_{\lambda_{1}}^{-1}x_{\frac{1}{\sqrt{2}}(\epsilon_{l-j+1}+\epsilon_{j})}(-1)\cdots x_{\frac{1}{\sqrt{2}}(\epsilon_2+\epsilon_{l-1})}(-1)$\\
$x_{\frac{1}{\sqrt{2}}(\epsilon_1+\epsilon_l)}(-1)v_{\Lambda_{l-j+2}}
$ is a weight vector with weight $\Lambda_{j}$, since we have
$$\alpha_j^\vee (0) e_{\lambda_{1}}^{-1}u=e_{\lambda_{1}}^{-1}\alpha_j^\vee (0) u=e_{\lambda_{1}}^{-1}u,$$
$$\alpha_1^\vee (0) e_{\lambda_{1}}^{-1}u=e_{\lambda_{1}}^{-1}(\alpha_1^\vee (0) -c) u=0,$$
$$\alpha_i^\vee (0) e_{\lambda_{1}}^{-1}u= e_{\lambda_{1}}^{-1}\alpha_i^\vee (0) u=0, \ \ \text{for} \ \ i \neq 1,j,$$
$$x_{-\theta}(1) e_{\lambda_{1}}^{-1}u=e_{\lambda_{1}}^{-1}x_{-\theta}(2) u=0.$$
From $x_{\alpha_1}(-1) x_{\alpha_1}(-1)v_{\Lambda_{l-j+2}}=0$ follows that $x_{\alpha_1}(-1) x_{\frac{1}{\sqrt{2}}(\epsilon_1+\epsilon_l)}(-1)v_{\Lambda_{l-j+2}}=0,$ so we have 
$$x_{\alpha_1}(0)  e_{\lambda_{1}}^{-1}u=0.$$
We also have
$$ x_{\alpha_j}(0) e_{\lambda_{1}}^{-1}u= e_{\lambda_{1}}^{-1} x_{\alpha_j}(0) u=0,$$
and
$$ x_{\alpha_i}(0) e_{\lambda_{1}}^{-1}u= e_{\lambda_{1}}^{-1} x_{\alpha_i}(0) u=0,$$
for $i \neq j, 1$. The last statement is true also in the case when $x_{\alpha_i}(0)$ doesn't commute with monomials in $u$. In this case, by induction on $j$ follows that $x_{\alpha_i}(0)u=u'x_{\frac{1}{\sqrt{2}}(\epsilon_{a}+\epsilon_{b})}(-1)v_{\Lambda_{l-j+2}}=0$, since the restriction of $L(\Lambda_l)$ on $\widetilde{\mathfrak{sl}}_2(\frac{1}{\sqrt{2}}(\epsilon_{a}+\epsilon_{b}))$, for any $a, b \in \{1, \ldots, l\}$, is a level two standard $A_1^{(1)}$-module of highest weight $2\Lambda_1$. Hence \eqref{sc4} holds and $L(\Lambda_j)^{(\lambda_1)} \cong L(\Lambda_{l-j+2}) $.
\end{proof}

By Proposition 2.4 in \cite{Li1} there is an intertwining operator of type $\binom{L(\Lambda_j)^{(\lambda_1)} }{L(\Lambda_j)   \,\,  L(\Lambda_0)^{(\lambda_1)}}$, which we will denote by $I_2(\cdot, z)$, such that
\begin{eqnarray}\nonumber
I_2(v_{\Lambda_j}, z)e_{\lambda_{1}}v_{\Lambda_0}&=&e_{\lambda_{1}}I_1(\Delta(\lambda_1,z)v_{\Lambda_j}, z) v_{\Lambda_0}\\
\nonumber
&=&e_{\lambda_{1}}\text{exp}(zL(-1))z^{\left< \lambda_1, \Lambda_j\right>}v_{\Lambda_j}\\
\nonumber
&=& z^{\left< \lambda_1, \Lambda_j\right>}\left[ e_{\lambda_{1}} v_{\Lambda_j}+e_{\lambda_{1}} v_{\Lambda_j}zL(-1)+\cdots \right].
\end{eqnarray}
Denote by $I_3(\cdot, z)$ the intertwining operator of type $\binom{L(\Lambda_j)^{(\lambda_1)} }{L(\Lambda_0)^{(\lambda_1)}   \,\,  L(\Lambda_j)}$ and by $A_{\lambda_1}$ the coefficient of $(-z)^{\left< \lambda_1, \Lambda_j\right>}$ in $I_3(e_{\lambda_1}v_{\Lambda_0}, z)v_{\Lambda_j}$. From Lemma \ref{lemoi} follows 
\beq \label{op2}
A_{\lambda_1}v_{\Lambda_j}=e_{\lambda_1}v_{\Lambda_j}.
\eeq
From the commutator formula (cf. formula (2.13) of \cite{Li3}) we have
\beq\nonumber
\left[ x_{\alpha_i}(m), I_3(e_{\lambda_1}v_{\Lambda_0}, z)\right]=\sum_{t\geq 0}\binom{m}{t}I_3(x_{\alpha_i}(t)e_{\lambda_1}v_{\Lambda_0}, z)=0.
\eeq

In the case when $ \widetilde{\mathfrak{g}}$ is of type $F_4^{(1)}$ or of type $G_2^{(1)}$, recall from \cite{BK} and \cite{Bu3} the constant term $A_{\lambda_1}$ of the operator $x_{\theta}(z)$. We will use the same symbol to denote the coefficient of $z^{-1}$ in $z^{-2}x_{\theta}(z)$, i.e. 
\beq\label{sc5}
A_{\lambda_1} = \text{Res}_z z^{-2}x_{\theta}(z)=x_{\theta}(-2).
\eeq

Now, consider the action of the operator 
$$(A_{\lambda_1})_s:=\underbrace{1\otimes\cdots \otimes  1}_{k-s \ \text{factors}} \otimes A_{\lambda_1}\underbrace{\otimes 1\otimes \cdots \otimes 1}_{ s-1 \ \text{factors}},
$$
for $s=n_{1,1}$, where $A_{\lambda_1}$ is as above, on the vector $\pi_{\Dc}bv_{ \Lambda }$, where $bv_{ \Lambda }\in \mathcal{B}_W$ and $b$ is of dual-charge type $\Dc$. Since $A_{\lambda_1}$ commutes with the action of quasi-particles, it follows that the image $(A_{\lambda_1})_s(\pi_{\Dc}bv_{ \Lambda})
$ is the coefficient of the variables 
$z_{r_l^{(1)},l}^{-m_{r_l^{(1)},l}-n_{r_l^{(1)},l}} \cdots   z_{1,1}^{-m_{1,1}-n_{1,1}}$  
in
\begin{equation}\label{sc7}
(A_{\lambda_1})_s\pi_{\Dc}x_{n_{r_{l}^{(1)},l}\alpha_{l}}(z_{r_{l}^{(1)},l})\cdots x_{n_{1,1}\alpha_{1}}(z_{1,1})v_{ \Lambda}.\end{equation}
Using \eqref{eq:p2} follows that in the $s$-th tensor factor (from the right) of \eqref{sc7}, we have $F_s  A_{\lambda_1} v_{ \Lambda_{j^s}}$, where 
$$F_s:=  x_{n_{r^{(\nu_l(k-1)+1)}_{l},l}^{(s)}\alpha_{l}}(z_{r_{l}^{(\nu_l(k-1)+1)},l})\cdots  x_{n_{r^{(\nu_lk)}_{l},l}^{(s)}\alpha_{l}}(z_{r_{l}^{(\nu_lk)},l})\cdots   x_{n_{1,l}^{(s)}\alpha_{l}}(z_{1,l}) \cdots$$
$$ \ \ \ \ \ \ \ \ \ \ \ \ \ \ \ \   x_{n_{r^{(s)}_{1},1}^{(s)}\alpha_{1}}(z_{r_{1}^{(s)},1})\cdots    x_{n{_{1,1}^{(s)}\alpha_{1}}}(z_{1,1}) .$$
Using \eqref{op2}, in the case of affine Lie algebras of type $B_l^{(1)}$ or of type $C_l^{(1)}$, we rewrite the $s$-th tensor factor (from the right) of \eqref{sc7} as
\beq\label{sc8}
F_s  A_{\lambda_1} v_{ \Lambda_{j^s}}=e_{\lambda_1}F_s  v_{ \Lambda_{j^s}}z_{r_{1}^{(s)},1}\cdots  z_{1,1}.
\eeq
From \eqref{sc8}, now follows
\beq\label{sc9}
(A_{\lambda_1})_s(\pi_{\Dc}bv_{ \Lambda})=(e_{\lambda_1})_s(\pi_{\Dc}b^+v_{ \Lambda}),
\eeq
where 
$$(e_{\lambda_1})_s:=\underbrace{1\otimes\cdots \otimes  1}_{k-s \ \text{factors}} \otimes e_{\lambda_1}\underbrace{\otimes 1\otimes \cdots \otimes 1}_{ s-1 \ \text{factors}},
$$
and where 
\beq\label{sc90}
b^+ =b_{\alpha_l}\,\cdots \,b_{\alpha_2}b^+_{\alpha_1}\eeq
with
\begin{eqnarray}\label{sc91}b^+_{\alpha_1} &=x_{n_{r^{(1)}_{1},1}\alpha_{1}}(m_{r^{(1)}_{1},1})\cdots  x_{n_{r^{(s+1)}_{1}+1,1}\alpha_{1}}(m_{r^{(s+1)}_{1},1})\\
\nonumber
& x_{n_{r^{(s)}_{1},1}\alpha_{1}}(m_{r^{(s)}_{1},1}+1)\cdots  x_{n_{2,1}\alpha_{1}}(m_{1,1}+1).
\end{eqnarray}
In the case of affine Lie algebras of type $F_4^{(1)}$ or of type $G_2^{(1)}$, the role of $e_{\lambda_1}$ will play Weyl group translation operator $e_{\theta}$, which we will introduce in the next section.

\subsection{Weyl group translation operators}\label{ss3.3}
Denote by $e_{\alpha}$ the Weyl group translation operator 
$$
 e_{\alpha}=\exp  x_{-\alpha}(1)\exp  (- x_{\alpha}(-1))\exp  x_{-\alpha}(1) \exp x_{\alpha}(0)\exp   (-x_{-\alpha}(0))\exp x_{\alpha}(0) $$
for every root $\alpha$, (cf. \cite{K}). We will use the following property of the Weyl group translation operator $e_{\alpha}$:
\beq\label{wg1}
x_{\beta}(j)e_{\alpha}=e_{\alpha}x_{\beta}(j+\beta(\alpha\sp\vee)) \quad \text{for all }   \alpha,\beta \in R\text{ and } j \in \mathbb{Z}.
\eeq

Let  $ \widetilde{\mathfrak{g}}$ be of type $F_4^{(1)}$ or of type $G_2^{(1)}$, and assume that $\alpha=\theta$. Since we have
\beq\label{wg2}
e_{\theta}v_{\Lambda_j}=x_{\theta}(-2) v_{\Lambda_j},
\eeq
we rewrite the $s$-th tensor factor (from the right) of \eqref{sc7} as
\beq\label{wg2}
F_s  A_{\lambda_1} v_{ \Lambda_{j^s}}=e_{\theta}F_s  v_{ \Lambda_{j_s}} z_{r_{1}^{(s)},1} \cdots z_{1,1}.
\eeq
Now we have
\beq\label{wg3}
(A_{\lambda_1})_s(\pi_{\Dc}bv_{ \Lambda})=(e_{\theta})_s(\pi_{\Dc}b^+v_{ \Lambda}),
\eeq
where 
$$(e_{\theta})_s:=\underbrace{1\otimes\cdots \otimes  1}_{k-s \ \text{factors}} \otimes e_{\theta}\underbrace{\otimes 1\otimes \cdots \otimes 1}_{ s-1 \ \text{factors}},
$$
and where $b^+$ is as in \eqref{sc90} and \eqref{sc91}. 

If we continue to apply the procedure of action of operators $(A_{\lambda_1})_s$ and $(e_{\lambda_1})_s$ (or $(A_{\lambda_1})_s$ and $(e_{\theta})_s$ in the case of affine Lie algebras of type $F_4^{(1)}$ or of type $G_2^{(1)}$), on the $b^+v_{ \Lambda}$, after finitely many steps we will obtain the monomial vector
\beq\nonumber
\tilde{b}v_{ \Lambda}=b_{\alpha_{l}}\cdots b_{\alpha_{2}}{\tilde{b}}_{\alpha_{1}}v_{ \Lambda},\eeq
where
\begin{align}\nonumber
{\tilde{b}}_{\alpha_{1}}&= x_{n_{r_{1}^{(1)},1}\alpha_1}(\tilde{m}_{r_{1}^{(1)},1})\cdots x_{n_{r_{1}^{(s)}+1,1}\alpha_{1}} (\tilde{m}_{r_{1}^{(s)}+1,1})x_{n_{r_{1}^{(s)},1}\alpha_{1}} (\tilde{m}_{r_{1}^{(s)},1})\cdots x_{n_{1,1}\alpha_{1}} (\tilde{m}_{1,1})
\\
\nonumber
&= x_{n_{r_{1}^{(1)},1}\alpha_1}(m_{r_{1}^{(1)},1})\cdots x_{n_{r_{1}^{(s)}+1,1}\alpha_{1}} (m_{r_{1}^{(s)}+1,1})x_{n_{r_{1}^{(s)},1}\alpha_{1}} (m_{r_{1}^{(s)},1}-m_{1,1}-s)\cdots x_{n_{1,1}\alpha_{1}} (-s).
\end{align}
Note that the quasi-particle monomial $\tilde{b}$ has the same charge-type and the dual charge-type as $b^+$ and belongs to $B_W$. 

We use the fact that
\beq \label{wg4}
e_{\alpha_1}v_{ \Lambda_j }=Cx_{\alpha_1}(-1)v_{\Lambda_j},
\eeq
where $C \in \mathbb{C}\setminus \{0\}$. Hence,  the vector $\pi_{\Dc}\,b^+v_{ \Lambda }$  equals the coefficient of the variables
\beq\label{wg5}
z_{r_l^{(1)},l}^{-m_{r_l^{(1)},l}-n_{r_l^{(1)},l}} \cdots z_{r_1^{(1)},1}^{-m_{r_1^{(1)},1}-n_{r_1^{(1)},1}}\cdots  z_{r_{1}^{(s)}+1,1}^{-m_{r_1^{(s)}+1,1}-n_{r_1^{(1)}+1,1}} z_{r_{1}^{(s)},1}^{-m_{r_1^{(s)},1}-n_{r_1^{(s)},1}+s}\cdots z_{1,1}^{-m_{1,1}} 
\eeq
 in 
\beq\label{wg5}
 C \ \pi_{\Dc} x_{n_{r_{l}^{(1)},l}\alpha_{l}}(z_{r_{l}^{(1)},l})  \cdots x_{n_{2,1}\alpha_{1}}(z_{2,1}) 
\ (1^{\otimes (k-s)} \otimes e_{\alpha_1}^{\otimes s})v_{ \Lambda}.
\eeq

By shifting the operator  $(1^{\otimes (k-s)} \otimes e_{\alpha_1}^{\otimes s})$ all the way to the left in \eqref{wg5}, and by dropping it, from \eqref{wg1} we get $\pi_{\Dc'}\,b'v_{ \Lambda}$, where 
\beq\label{wq6}b'=b_{\alpha_{l}}\cdots \, b'_{\alpha_{2}}\,b'_{\alpha_{1}}\eeq
for
\begin{align*}
b'_{\alpha_{1}}&=x_{n_{r^{(1)}_{1},1}\alpha_{1}}(\tilde{m}_{r^{(1)}_{1},1}+2n_{r^{(1)}_{1},1})\cdots   x_{n_{2,1}\alpha_{1}}(\tilde{m}_{2,1}+2n_{2,1}),\\
b'_{\alpha_{2}}&=x_{n_{r^{(1)}_{2},2}\alpha_{2}}(m_{r^{(1)}_{2},2}-n^{(1)}_{r_{2}^{(1)},2}-\cdots-n^{(s)}_{r_{2}^{(1)},2})\cdots x_{n_{1,2}\alpha_{2}}(m_{1,2}-n^{(1)}_{1,2}-\cdots-n^{(s)}_{1,2}).
\end{align*}
Note, that the dual charge-type $\Dc'$  of $b'$  equals
$$\Dc'=\big(r^{(1)}_{l}, \ldots , r^{( k_{\alpha_l})}_{l};\cdots ; r^{(1)}_{2}, \ldots , r^{(k_{\alpha_2})}_{2};\, r^{(1)}_{1}-1,\ldots, r_1^{(n_{1,1})}-1,\underbrace{0, \ldots, 0}_{k-s}\big).
$$
Finally, using the same arguments as in \cite{Bu1,Bu2, Bu3, BK} one can check that $b'$ belongs to $B_W$.

\subsection{Proof of linear independence}\label{ss3.4}
Assume that we have a relation of linear dependence between elements of $\mathcal{B}_W$
\begin{equation}\label{eq:d1}
\sum_{a \in A}
c_{a}b_av_{\Lambda}=0, 
\end{equation}
where $A$ is a finite non-empty set and $c_{a} \neq 0$ for all $a \in A$. Furthermore, assume that all $b_a$ have the same color-type $ \left(n_{l},\ldots, n_{1}\right)$. Let $a_0 \in A$ be such that $b_{a_0} < b_a$ for all $a \in A$, $a \neq a_0$. Let $b_0$ be of charge-type $\Cc$ as in \eqref{charge-type} and dual-charge-type $\Dc$ as in \eqref{dual-charge-type}.

On (\ref{eq:d1}) we act with the projection $\pi_{\Dc}$ 
$$\pi_{\Dc}: W_{L(\Lambda)}\rightarrow  {W_{L(\Lambda_{j^k})}}_{(\mu^{(k)}_{l};\ldots;\mu_{1}^{(k)})}\otimes \cdots \otimes  {W_{L(\Lambda_{j^1})}}_{(\mu^{(1)}_{l};\ldots;\mu_{1}^{(1)})},$$  
where $j^t \in \{0,j\}$ and $\mu^{(t)}_{i}$, $1 \leq t\leq k$ is as in \eqref{proj0}. From the definition of projection, it follows that by $\pi_{\Dc}$ all monomial vectors $b_av_{ \Lambda}$ with monomials $b_a$ which have higher charge-type than $\Cc$ with respect to \eqref{order2} will be mapped to zero-vector. Therefore, we assume that in
\beq\label{eq:d2}
\sum_{a \in A}
c_{a}\pi_{\Dc}b_av_{ \Lambda}=0, 
\eeq 
all monomials $b_a$ are of charge-type $\Cc$.

In the case when $ \widetilde{\mathfrak{g}}$ be of type $B_l^{(1)}$ or of type $C_l^{(1)}$ and $\Lambda=k_0\Lambda_0+k_1\Lambda_1$ we use the fact that $v_{\Lambda_1}=e_{\lambda_1}v_{\Lambda_0}$. Now, from \eqref{eq:d2} we have
\beq\nonumber
0=\sum_{a \in A}
c_{a}\pi_{\Dc}b_a(e_{\lambda_1}v_{\Lambda_0})^{\otimes k_j}\otimes v_{ \Lambda_0}^{\otimes k_0}=(e_{\lambda_1}^{\otimes k_j}\otimes 1^{\otimes k_0})\sum_{a \in A}
c_{a}\pi_{\Dc}b^{+}_a v_{ \Lambda_0}^{\otimes k}.
\eeq
If we drop the operator $e_{\lambda_1}^{\otimes  k_j}\otimes 1^{\otimes k_0}$, we will get 
\beq \label{eq:d40}
\sum_{a \in A}
c_{a}\pi_{\Dc}b^{+}_a v_{ \Lambda_0}^{\otimes k}=0,
\eeq
where $b^{+}_a $ is of the form as in \eqref{sc13} and \eqref{sc14}. 
So, in the case of $B_l^{(1)}$ or of type $C_l^{(1)}$ and $\Lambda=k_0\Lambda_0+k_1\Lambda_1$, we have $c_a=0$, and the assertion of the Theorem \ref{t1} follows.
 
Now, let $ \widetilde{\mathfrak{g}}$ be of type $B_l^{(1)}$, $C_l^{(1)}$, $F_4^{(1)}$ or of type $G_2^{(1)}$ and $\Lambda=k_0\Lambda_0+k_j\Lambda_j$, $j \neq 1$. On (\ref{eq:d2}) apply the procedure described in Sections \ref{ss3.2} and \ref{ss3.3} until all quasi-particles of color $1$ are removed from the  summand $c_{a_0}\pi_{\Dc}\, b_{a_0}  v_{\Lambda}$.
This also removes all quasi-particles of color $1$ from other summands, so that \eqref{eq:d2}  becomes
\beq\label{eq:d41}
\sum_{a \in A}
c'_{a}\pi_{\Dc}b'_a v_{ \Lambda}=0,
\eeq
where $b'_a $ are of the form as in \eqref{wq6} and  scalars $c'_{a}\neq 0$. The summation in \eqref{eq:d41} goes over all $a\neq a_0$ such that $b_{a}(\alpha_1)= b_{a_0}(\alpha_1)$ since the summands such that  $b_{a_0}(\alpha_1)< b_a(\alpha_1)$ will be annihilated in the process.
	
In the case of $B_l^{(1)}$ monomial vectors in \eqref{eq:d41} can be realized as elements of the principal subspace $W_{L(k_0\Lambda_0+k_l\Lambda_l)}$ of the affine Lie algebra of type $B_{l-1}^{(1)}$. In particular, when $\widetilde{\mathfrak{g}}$ is of type $B_2^{(1)}$, with the described procedure we get monomial vectors which can be realized as elements of $W_{L((2k_0+k_2)\Lambda_0+(2k-2k_0-k_2)\Lambda_2)}$ of the affine Lie algebra of type $A_{1}^{(1)}$. In the case of $C_l^{(1)}$, monomial vectors in \eqref{eq:d41} can be realized as elements in $W_{L((2k_0+k_j)\Lambda_0+(2k-2k_0-k_j)\Lambda_j)}$  of the affine Lie algebra of type $A_{l-1}^{(1)}$. In the case of $F_4^{(1)}$, monomial vectors in \eqref{eq:d41} can be realized as elements in $W_{L(k_0\Lambda_0+k_3\Lambda_3)}$  of the affine Lie algebra of type $C_3^{(1)}$, and in the case of $G_2^{(1)}$ monomial vectors in \eqref{eq:d41} we realize as elements of $W_{L((3k_0+2k_2)\Lambda_0+(3k-3k_0-2k_2)\Lambda_2)}$ of the affine Lie algebra of type $A_{1}^{(1)}$. For all of these cases we can use Georgiev argument on linear independence from \cite{G1}, so by  proceeding inductively on charge-type (and on $l$ for the case of $B_l^{(1)}$) we get $c_{a}=0$ and the desired theorem follows.

\section{Characters of principal subspaces}
Character $\ch W_{L(\Lambda)}$ of the principal subspace $W_{L(\Lambda)}$ is defined by 
 $$
 \ch W_{L(\Lambda)}=\sum_{m,n_1,\ldots,n_l\geqslant 0} 
\dim (W_{L(\Lambda)})_{-m\delta +\Lambda+n_1\alpha_1 +\ldots + n_l\alpha_l}\, q^{m}y^{n_1}_{1}\cdots y^{n_l}_{l},
$$
where $q,  y_1,\ldots, y_l$ are formal variables and $(W_{L(\Lambda)})_{-m\delta +\Lambda +n_1\alpha_1 +\ldots + n_l\alpha_l}$ denote the  weight subspaces of $W_{L(\Lambda)}$ of weight $-m\delta +\Lambda +n_1\alpha_1 +\cdots + n_l\alpha_l$ with respect to the Cartan subalgebra $\widetilde{\mathfrak{h}} =\mathfrak{h}\oplus \mathbb{C}c\oplus \mathbb{C}d$ of $\widetilde{\mathfrak{g}}$.

As in \cite{Bu1, Bu2, Bu3, BK, G1}, to determine the character of $W_{L(\Lambda)}$, we write conditions on energies of quasi-particles of the set $B_{W_{L(\Lambda)}}$ in terms of dual-charge-type elements $r_i^{(s)}$. For a fixed color-type $(n_l; \ldots ; n_1)$, charge-type
$$\Cc =\left( n_{r_l^{(1)},l}, \ldots,  n_{1,l}; \ldots ; n_{r_1^{(1)},1}, \ldots,  n_{1,1}\right)$$ and dual-charge-type
$$\Dc =\left(r^{(1)}_l, \ldots , r^{(k_{\alpha_l})}_l; \ldots ;r^{(1)}_1, \ldots , r^{(k_{\alpha_1})}_1\right),$$
a straightforward calculation shows
\begin{equation} \label{uvjet1}
\sum_{p=1}^{r_{i}^{(1)}} (2(p-1)n_{p,i}+n_{p,i})= \sum_{t=1}^{k_{\alpha_i}}r^{(t)^{2}}_{i}  \quad\text{for }i=1,\ldots , l,\end{equation}
and
\begin{equation} \label{uvjet2}
\sum_{p=1}^{r^{(1)}_{i}}\sum_{q=1}^{r^{(1)}_{i-1}}\mathrm{min}\{\frac{k_{\alpha_i}}{k_{\alpha_{i-1}}}n_{q,i-1},n_{p,i}\}=\sum_{t=1}^{k}\sum_{p=0}^{\nu_i-1}r_{i-1}^{(t)}
r_{i}^{\left(\nu_i t -p\right)}  \quad\text{for }i=2,\ldots , l,\end{equation}
(cf. \cite{Bu1, Bu2, Bu3, BK, G1}). We also have  
\begin{equation}\label{uvjet3}
\sum_{p=1}^{r^{(1)}_{i}}\sum_{t=1}^{n_{p,i}}\delta_{i,j_t}=\sum_{t=1}^{k_{\alpha_i}}r_i^{(t)}\delta_{i,j_t}=\sum_{t=\nu_j k_0+(\nu_j-1)k_j+1}^{k_{\alpha_j}}r_j^{(t)} .
\end{equation}  
The last three identities, difference conditions \eqref{rel7}--\eqref{rel8} and the formula
$$
\frac{1}{(q)_r}=\sum_{n\geqslant 0}p_r(n)q^n,
$$
where $p_r(n)$ denotes the number of partitions of $n$ with at most $r$ parts, therefore imply
\begin{thm}\label{thm_karakter} 
Set  $n_i=\sum_{t=1}^{k_{\alpha_i}}r_i^{(t)}$  for $i=1,\ldots ,l$.
For any rectangular weight $\Lambda = k_0\Lambda_0 + k_j\Lambda_j$ of level $k = k_0 + k_j$  we have
$$ \ch W_{L(\Lambda)}= 
\sum_{\substack{r_{1}^{(1)}\geqslant \cdots\geqslant r_{1}^{(k_{\alpha_1})}\geqslant 0\vspace{-5pt}\\ \vdots\vspace{-2pt} \\r_{l}^{(1)}\geqslant \cdots\geqslant r_{l}^{(k_{\alpha_l})}\geqslant 0}}
\frac{q^{\sum_{i=1}^l\sum_{t=1}^{k_{\alpha_i}}r_i^{(t)^2}
-\sum_{i=2}^{l}\sum_{t=1}^{k}
\sum_{p=0}^{\nu_i-1}r_{i-1}^{(t)}
r_{i}^{\left(\nu_i t -p\right)} +\sum_{t=\nu_j k_0+(\nu_j-1)k_j+1}^{k_{\alpha_j}}r_j^{(t)} }}
{\prod_{i=1}^{l}(q;q)_{r^{(1)}_{i}-r^{(2)}_{i}}\cdots (q;q)_{r^{(k_{\alpha_i})}_{i}}}\,
\prod_{i=1}^{l}y^{n_i}_{i},
$$
where $(a;q)_r=\prod_{i=1}^r (1- aq^{i-1}) \ \ \text{for} \ \ r\geqslant 0$.
\end{thm}

\section*{Acknowledgement}
I am very grateful to Slaven Ko\v zi\' c and Mirko Primc for their help, support and valuable comments during the preparation of this work.

This work is partially supported by the QuantiXLie Centre of Excellence, a project cofinanced by the Croatian Government and European Union through the European Regional Development Fund - the Competitiveness and Cohesion Operational Programme (Grant KK.01.1.1.01.0004) and by Croatian Science Foundation under the project 8488.

\end{document}